\documentclass[a4paper,11pt,reqno]{amsart}

\usepackage{url}
\usepackage{geometry}
\usepackage{mathrsfs}
\usepackage{comment}
\usepackage{verbatim}
\usepackage[usenames,dvipsnames]{xcolor}
\usepackage[shortlabels]{enumitem}
\usepackage{marginnote}

\usepackage[T1]{fontenc}
\usepackage[english]{babel}
\usepackage[utf8]{inputenc}

\usepackage{hyperref}
\usepackage{url}
\hypersetup{%
    backref = page,%
    colorlinks,%
    linkcolor = {red!60!black},%
    citecolor = {green!60!black},%
    urlcolor = {blue!60!black}%
}

\usepackage{graphicx}
\graphicspath{{../figures/}}
\usepackage{tikz}
\usetikzlibrary{decorations.markings}
\tikzset{->-/.style={decoration={
  markings,
  mark=at position #1 with {\arrow{>}}},postaction={decorate}}}
\usepackage[symbol]{footmisc}

\usepackage{amsmath}
\usepackage{amssymb}
\usepackage{amscd}
\usepackage{amsthm} 

\theoremstyle{plain}

\newtheorem{theorem}{Theorem}[section]
\newtheorem{thm}{Theorem}[section]
\newtheorem{lemma}[thm]{Lemma}
\newtheorem{corollary}[thm]{Corollary}
\newtheorem{cor}[thm]{Corollary}
\newtheorem{prop}[thm]{Proposition}
\newtheorem{fact}[thm]{Fact}

\theoremstyle{definition}

\newtheorem{conjecture}[thm]{Conjecture}

\theoremstyle{remark}

\newcommand{\defi}[1]{\emph{\color{red!50!black}#1}}

\def\NN{\mathbb N}
\def\semidegree{\delta^{0}}
\def\Out{\textsc{Out}}
\def\In{\textsc{In}}

\let\eps=\varepsilon
\let\theta=\vartheta
\let\phi=\varphi

\begin{document}

\title{Antidirected trees in directed graphs}
\date{\today}

\author[G. Kontogeorgiou]{George Kontogeorgiou}
\address{Centro de Modelamiento Matemático (CNRS IRL2807), Universidad de Chile, Santiago, Chile}
\email{gkontogeorgiou@dim.uchile.cl}

\author[G. Santos]{Giovanne Santos}
\address{Departamento de Ingeniería Matemática, Universidad de Chile, Santiago, Chile}
\email{gsantos@dim.uchile.cl}

\author[M. Stein]{Maya Stein}
\address{Departamento de Ingeniería Matemática y Centro de Modelamiento Matemático (CNRS IRL2807), Universidad de Chile, Santiago, Chile}
\email{mstein@dim.uchile.cl}
\thanks{George Kontogeorgiou was supported by ANID grant CMM Basal FB210005. Giovanne Santos was supported by ANID Becas/Doctorado Nacional 21221049. Maya Stein was supported by ANID Regular Grant 1221905 and by ANID grant CMM Basal FB210005.}


\begin{abstract}
By the well-known Koml\'os-S\'ark\"ozy-Szemer\'edi (KSS) theorem, any $n$-vertex graph of  minimum degree $n/2+o(n)$  contains all bounded degree trees of   order $n$. Several variants of this result have been put forward recently, where the tree is allowed to be of smaller order than the host graph. In these results, the host graph also obeys a maximum degree condition, which is necessary. In a different vein, the KSS theorem was generalised to digraphs. 

We bring these two directions together by establishing   minimum and maximum degree bounds for digraphs that ensure the containment of oriented trees of smaller order. Our result, Theorem \ref{tree}, is restricted to balanced antidirected trees of bounded degree but may hold for other oriented trees as well. 
 The full statement of Theorem \ref{tree} can be found below, we will limit ourselves here to state three important implications. 

  First, the result implies that, asymptotically, every large digraph of minimum semidegree above $2k/3$ having vertices of out-degree and in-degree  above $k$ contains each large  balanced antidirected bounded-degree  tree with~$k$ arcs. These semidegree bounds are essentially best possible, and are identical to bounds in corresponding results for graphs.

  Second, Theorem \ref{tree} implies that 
  asymptotically, every large digraph of minimum semidegree above $3k/5$ having vertices of out-degree and in-degree  above $2k$ contains each large  balanced antidirected bounded-degree tree with~$k$ arcs. A bound of $k/2$ instead of $3k/5$ on the minimum semidegree would be tight and is known to be sufficient in corresponding results for graphs and  for oriented graphs.

  Thirdly, our result gives a version of the digraph KSS theorem for smaller trees. We prove that for all $k$ and $\varepsilon$ there is a $C$ such that  every large digraph of minimum semidegree above~$(1+\varepsilon)k/2$ having vertices of out-degree and in-degree  above $Ck$ contains each large balanced antidirected bounded-degree tree with~$k$ arcs.  An exact version of our result for graphs and trees of any maximum degree has been proved recently by Hyde and Reed.

\end{abstract}

\maketitle

\section{Introduction}
\label{sec:intro}

 Subgraph containment problems lie at the heart of extremal graph theory. In order to ensure containment of a given subgraph, one of the most widely used conditions on the host graph is a bound on its minimum degree. A famous example  is Dirac's  theorem on Hamilton cycles.
 In this paper we focus on tree subgraphs. One of the best-known theorems in this direction  is the Koml\'os--S\'ark\"ozy--Szemer\'edi theorem on spanning trees of bounded maximum degree.

\begin{theorem}[Koml\'os, S\'ark\"ozy and Szemer\'edi~\cite{KSS2001}]\label{thm:KSS}
    For every $\gamma>0$, there exist $n_0\in\mathbb N$ and~$c>0$ such that every graph $G$ on $n\ge n_0$ vertices with
    $\delta(G)\ge (1+\gamma)\frac{n}{2}$ contains every~$n$-vertex tree $T$ with $\Delta(T)\le c\frac{n}{\log n}$.
\end{theorem}

The bounds in Theorem~\ref{thm:KSS} are essentially best possible. In particular, the bound on $\Delta(T)$ in Theorem~\ref{thm:KSS}  is best possible up to the constant $c$~\cite{KSS2001}. If we wish to find {\it all} trees $T$ on $n$-vertices as spanning trees of $G$, then  clearly, $G$  needs to have a vertex of degree at least~$n-1$, because~$T$ could be a star. The existence of just one such vertex is indeed sufficient if we also change the bound on $\delta(G)$ from Theorem~\ref{thm:KSS} to about $2n/3$. This is the content of the following result due to Reed and the last author.
 
\begin{theorem}[Reed and Stein~\cite{RS19a, RS19b}]
\label{bruce}
There is $n_0\in\mathbb N$ such that for every $n\ge n_0$, every $n$-vertex graph with $\delta(G)\ge \lfloor{2(n-1)/3}\rfloor$ and  $\Delta(G)\ge n-1$ contains every $n$-vertex tree. 
\end{theorem}

We are interested in finding smaller trees with corresponding weaker bounds on $\delta(G)$. One of our motivations is the following conjecture.

\begin{conjecture}[Havet, Reed, Wood and Stein~\cite{havet}]\label{23conj}
Every graph of minimum degree at least~${\lfloor 2k/3 \rfloor}$ and maximum degree at least~$k$ contains every $k$-edge tree. 
\end{conjecture}

The bound on $\delta(G)$  in Conjecture~\ref{23conj} is best possible, as witnessed by the following example: let $T$ be any $k$-edge tree having a vertex $v$ such that $T-v$ has three components of exactly the same size. Let $G$ be obtained by taking two cliques of size $2k/3 -1$ and adding a new vertex adjacent to all  vertices of both cliques. Then $T$ is not a subgraph of $G$.

An approximate version of  Conjecture~\ref{23conj}  for large bounded degree trees was proved by Besomi, Pavez-Sign\'e and the last author
 in~\cite{BPS1}. More generally, these authors proved the following.
 \begin{theorem}[Besomi, Pavez-Sign\'e and Stein~\cite{BPS1}]\label{bps}
For every $\gamma >0$ there is $n_0\in\mathbb N$ such that for all $n\ge n_0$ and all $k\ge \gamma n$ every $n$-vertex graph $G$ fulfilling
\begin{enumerate}[(a)]
    \item $\delta(G)\ge (1+\gamma){k}/{2}$ and $\Delta(G)\ge 2(1+\gamma)k$; or
     \item $\delta(G)\ge (1+\gamma){2k}/{3}$ and $\Delta(G)\ge (1+\gamma)k$
\end{enumerate} 
contains every $k$-edge tree $T$  with $\Delta(T)\le k^{{1}/{67}}$ as a subgraph. 
\end{theorem}
The bounds on $\delta (G)$ and on $\Delta(G)$ in Theorem~\ref{bps} are essentially best possible. 
In view of Theorem~\ref{bps}~(a) and in analogy to Conjecture~\ref{23conj}, one might suspect that every graph of minimum degree exceeding~$k/2$ and maximum degree at least~$2k$  contains every $k$-edge tree. This was conjectured\footnote{A version with $\ge k/2$ replacing $>k/2$  had been conjectured earlier~\cite{BPS1}, but  was disproved by Hyde and Reed~\cite{hydeReed}.} in~\cite{maya_womenLA}. Even if it is false, replacing $2k$ with another function of $k$ will make it correct:
 
\begin{theorem}[Hyde and Reed~\cite{hydeReed}]\label{hydeReed}
    There is a function $f$ such that every graph with $\delta(G)\ge k/2$ and  $\Delta(G)\ge f(k)$ contains every $k$-edge tree.
\end{theorem}

Let us now turn our attention to digraphs. Kathapurkar and Montgomery~\cite{km} recently extended Theorem~\ref{thm:KSS} verbatim to digraphs, allowing the tree to have any orientation, and replacing the minimum degree $\delta(G)$  with the {\it minimum semidegree} $\delta^0(D)$.
Similarly to graphs, if we wish to find a variant of their theorem  for trees of smaller order, we need to add the condition that the host digraph $D$ has vertices of large outdegree and vertices of large indegree. Define~$\Delta(D)$ as the largest integer $m$ such that $D$ has a vertex of outdegree at least $m$ and a vertex of indegree at least $m$.

Our main result is a version of Theorem~\ref{bps} for digraphs and balanced antidirected trees (where a {\it balanced antidirected tree} is an oriented tree without 2-directed paths and with the same number of vertices in each partition class). It reads as follows.

\begin{theorem}\label{tree}
    For every $\gamma>0$, $c\in\mathbb R$ and $\ell\in\mathbb N$ with $\ell\ge 2$, there exists $n_0\in\mathbb N$ such that for all $n\geq n_0$ and for all $k\ge \gamma n$, every digraph~$D$ on $n$ vertices fulfilling 
   $$\delta^0(D)\geq \left(\frac{\ell}{2\ell-1}+\gamma\right)k\text{ and }\Delta(D)\geq (1+\gamma)(\ell-1) k$$
    contains every balanced antidirected tree $T$ with $k$ arcs and with $\Delta(T)\leq (\log n)^c$. 
\end{theorem}

The bounds from Theorem~\ref{tree} may not seem aesthetically pleasing at first glance, but plugging in the concrete values $\ell=3$ and $\ell=2$ we obtain bounds that are very similar to the ones from Theorem~\ref{bps}:
 \begin{corollary}\label{corooo}
For every $\gamma >0$ and $c\in\mathbb R$ there is $n_0\in\mathbb N$ such that for all $n\ge n_0$ and all~$k\ge \gamma n$ every $n$-vertex digraph $D$ fulfilling
\begin{enumerate}[(a)]
    \item $\delta^0(D)\ge (1+\gamma){3k}/{5}$ and $\Delta(D)\ge 2(1+\gamma)k$; or
     \item $\delta^0(D)\ge (1+\gamma){2k}/{3}$ and $\Delta(D)\ge (1+\gamma)k$
\end{enumerate} 
 contains every balanced antidirected tree $T$ with $k$ arcs and with $\Delta(T)\leq (\log n)^c$.  
\end{corollary}
By the examples given above, and since any graph can be viewed as a digraph, the bound on~$\delta^0(D)$ in Corollary \ref{corooo}(b) is essentially best possible (and thus Corollary \ref{corooo}(b)  asymptotically solves Problem 6.13 from~\cite{maya_bcc_survey} for balanced antidirected trees of bounded degree in digraphs). For the same reasons, in both parts  the bounds on $\Delta (D)$ are essentially best possible, and in part (a) the bound on $\delta^0(D)$ needs to be at least about~$k/2$
.  

We do not know if Corollary \ref{corooo}(a) can be improved in that sense or if in the general digraph setting, a larger minimum semidegree is necessary. For certain classes of digraphs, $k/2$ is  sufficient, with better or no bounds on $\Delta (D)$: first, 
 Trujillo-Negrete and the last author showed in~\cite{ana2} that 
every digraph $D$ with $\delta^0(D) \ge k/2$ and $\Delta(D) \ge k$ and without  oriented 4-cycles  contains all oriented trees  with $k$ arcs.
Second,  Z\'arate-Guer\'en and the last author showed in~\cite{SZ24} that
for all $\gamma, c$, sufficiently large~$n$ and $k\geq \gamma n$,  every oriented graph $D$ on $n$ vertices with $\delta^0 (D)>(1 + \gamma)k/2$ contains  every balanced antidirected tree $T$ with $k$ arcs and with  $\Delta(T) \le (\log(n))^c$. 
The more general approach from the present  paper allows us to quickly reprove the latter result as Corollary~\ref{corollo}.
In Section \ref{overvi} we will explain why Corollary~\ref{corollo} is much easier to prove than Theorem~\ref{tree}.

If $\ell$ tends to infinity and $\gamma$ approaches~$0$, the  minimum semidegree bound of Theorem \ref{tree} tends to $k/2$, and the bound on $\Delta (D)$ goes to infinity. More precisely, we have the following corollary which follows from Theorem~\ref{tree} by taking {$\gamma=\eps/10$ and $\ell=\lceil\frac{10}{\eps^2}\rceil$.}


\begin{corollary}\label{corooo3}
For every {$0<\eps<1$}, $c\in\mathbb R$ there is $n_0\in\mathbb N$ such that for all $n\ge n_0$ and all $k\ge \eps n$ every $n$-vertex digraph $D$ with $\delta^0(D)\ge (1+\eps){k}/{2}$ and $\Delta(D)\ge k/\eps$
 contains every balanced antidirected tree $T$ with $k$ arcs and with $\Delta(T)\leq (\log n)^c$.  
\end{corollary}
Corollary \ref{corooo3} constitutes an asymptotic version of Theorem~\ref{hydeReed} for large bounded degree trees in digraphs.
Just as Theorem~\ref{hydeReed} can be viewed as variant of the KSS theorem for smaller trees, Corollary \ref{corooo3} is a variant of the digraph KSS theorem from \cite{km} for smaller trees (for a somewhat restricted class of trees).
For more context, history and open directions, we refer the interested reader to the survey~\cite{maya_bcc_survey} and the recent~\cite{maya_womenLA}.

We close the introduction by remarking that
although our proof strategy only seems to work for balanced antidirected trees of bounded degree, 
it may be that a variant of Theorem \ref{tree} or of any of its corollaries holds for other oriented trees. 
A quick idea of our proof method and its limitations is given in Section~\ref{overvi}.

\section{Overview of our proof and organisation of the paper.}\label{overv}
\subsection{Overview of the proof of Theorem \ref{tree}.}\label{overvi}
We will rely on the regularity method for digraphs as well as on existing results and methods for cutting up trees and embedding them into regular pairs. This has been done similarly for embedding bounded-degree trees in graphs in several papers, for instance in \cite{BPS1}, and it has been done recently for embedding antidirected trees in oriented graphs in \cite{SZ24}. Embedding into edges of the reduced digraph is straightforward for antidirected trees, however, finding the appropriate connections between the edges of the reduced digraph can be difficult. The connections we need are antidirected paths in the reduced digraph that run between   the edges we plan to use for the embedding.

 In \cite{SZ24}, where a similar result to our Corollary \ref{corooo}(a) was shown for oriented graphs, the reduced digraph is shown to have a sufficiently large matching whose edges are connected in the above sense (we call such a matching an {\it anticonnected matching}). 
 This was possible even with a bound of only about $k/2$ on the minimum semidegree and no requirement of having vertices of large semidegree because the host was an oriented graph which makes the problem much easier. For instance, in an oriented graph~$D$, just one vertex of semidegree at least~$k/2$ guarantees that~$D$ has more than $k$ vertices, while a digraph of minimum semidegree~$k/2$ can have as little as~$k/2+1$ vertices.

The main challenge in our proof is that we need to find a similar matching structure as the one used in \cite{SZ24}, but in a digraph $D$. In order to get a better grip on the possible connections $D$, we explore the {\it anticonnected components} of $D$. As hinted at above, the notion of anticonnectivity captures the meaning of vertices being connected by an antidirected path, in the same spirit as connectivity in undirected graphs represents being connected by an (undirected) path and strong connectivity in digraphs represents being connected by a directed path. The difference to both those settings is that anticonnected components have two types of vertices (just like antidirected paths), and distinct anticonnected components may intersect in certain ways. { Consider, for example, this
digraph: \raisebox{-0.3em}{\tikz[scale=0.1]{
    \begin{scope}
        \foreach \x\k in {0/0,3/1,6/2,9.5/3,13/4,16/5}{%
          \node[draw,color=black,fill=black,circle,inner sep=0.05cm] (a\k) at (\x,0) {};
          }
          \draw[thick,->] (a0) -- (a1);
          \draw[thick,<-] (a1) -- (a2);
          \draw[thick,->] (a2) -- (a3);
          \draw[thick,->] (a3) -- (a4);
          \draw[thick,<-] (a4) -- (a5);
          \draw[rounded corners,thin, fill=BrickRed, opacity=0.3] (-1.5,-2) rectangle (10.8,2);
          \draw[rounded corners,thin, fill=orange, opacity=0.3] (8.2,-2) rectangle (17.2,2);
    \end{scope}}}, which consists of two anticonnected components
  that intersect at a single vertex.}

Using the lower bound on the minimum semidegree from Theorem \ref{tree}, we show that the anticonnected components of $D$ each
 contain a suitable anticonnected matching. Our plan is to embed the tree into these matchings. 
 
 However, each of the anticonnected components of $D$ could be too small to host the tree $T$ on its own. This is where the vertex $u$ of large out- or in-degree comes into play. We prove that~$u$ has a non-negligible degree into at least $\ell$ anticonnected components of the reduced digraph (more precisely, $u$ will either send or receive some $\eps k$ arcs to/from suitable vertices of these components). We then cut the tree $T$ at a strategically chosen vertex $x$, embed $x$ into $u$, group  the trees of $T-x$ into $\ell$ families of certain total sizes, and embed each family into one of the anticonnected components seen by $u$. To be exact, we first find $x$ and then, depending on $x$, decide whether we take $u$ as the vertex of large out-degree or as the vertex of large in-degree of~$D$.

 Unfortunately, there is still one more obstacle to overcome: the components seen by $u$ could intersect. Thus after embedding part of the tree into the first components, parts of the other components may have become unusable. To address this issue, we have to take care  to group the trees of $T-x$ into families of the correct sizes, and we will choose the anticonnected components we use dynamically.

 Let us now address some of the limitations of our proof. First, our trees need to be antidirected, as otherwise it would not be possible to use the regular pairs for the embedding. Second, our trees need to be balanced as this allows us to control how we fill the  clusters of the anticonnected components. 
 Finally our trees need to have bounded degree because of our strategy of using the connecting paths between matching edges. All these restrictions (antidirected, balanced, bounded degree) come from our proof method and it is not impossible that they could be overcome with different methods.

 On the other hand, the requirement of $\Delta(D)\ge k$ is necessary, as otherwise the digraph could be too small to host the tree. In the case of $\delta^0(D)$ being bounded asymptotically by $k/2$, vertices of out-/in-degree at least $2k$ become necessary because of examples from \cite{BPS1}. These examples also show that the bounds on $\delta^0(D)$ in Corollary \ref{corooo}(b) and in Corollary~\ref{corooo3} are best possible, while those in Corollary \ref{corooo}(a) may be improvable.

\subsection{Organisation of the paper.}\label{orga}
Section \ref{sec:preliminaries} provides  a quick overview of digraph notation, some basic results on cutting and grouping trees, and the diregularity lemma.

In Section \ref{sec:anticomponents} we define anticonnected digraphs and components. 
We show that, under certain degree conditions, any large anticonnected digraph contains a large  matching  (see Lemma \ref{lemma:matchings}), we call such a matching an {\it anticonnected matching}. 

Then, in Section \ref{sec:embedding}, we show how to embed a tree or forest into a digraph $D$, if, after applying diregularity to $D$, the reduced digraph of $D$ has large  anticonnected component (see Lemma~\ref{lemma:emb_large}). For this, we first find a large antimatching in the anticonnected component and then embed most of the tree in the pairs corresponding to arcs of the antimatching, using only small bits of the tree  to connect parts embedded in distinct arcs of the antimatching.

Finally, in Section \ref{sec:three} we 
prove Theorem~\ref{tree} by showing that  $D$ contains a vertex $u$ such that $u$ has many neighbours in each of several (possibly overlapping) anticonnected components of the reduced digraph of $D-u$. We  then embed   $T$ into $D$, by embedding $r(T)$ in $u$ and using the result from Section~\ref{sec:three}. For this, we group the components of $T-r(T)$ into forests of adequate orders (see Lemma~\ref{cutTintoJs}), and then embed each of these forests in one of the anticonnected components of the reduced digraph of $D$. It will be important to carefully choose which group of forest is embedded into which antidirected component.

\section{Preliminaries}
\label{sec:preliminaries}

\subsection{Basic notation}
\label{ssec:notation}

Given a positive number $y$, we write $x \ll y$ to indicate that there is a positive number $x_0$, depending only on $y$, such that for all $0 < x \leq x_0$ the next statements hold.

A \defi{digraph} is a  pair $D=(V,A)$ where $V=V(D)$ is a set, and each element of the set $A=A(D)$ is an ordered pair of distinct elements of~$V$. The elements of~$V$ are the \defi{vertices} and the elements of~$A$ are the \defi{arcs} of $D$. For an arc $a=(u,v)$ we say that $u$ and $v$ are its \defi{ends}.   An \defi{oriented} graph is a digraph in which there is at most one arc between any pair of vertices. The \defi{distance} between two distinct vertices in $D$ is the length of a shortest oriented path (contained in $D$) that connects them. 
An \defi{oriented walk} in~$D$ is a  sequence~$W = (v_0,a_1,v_1,\dots,v_{\ell-1},a_\ell,v_\ell)$, the terms of which are alternately vertices and arcs of~$D$ such that~${a_i = (v_{i-1},v_{i})}$ or~${a_i = (v_{i},v_{i-1})}$ for all~${1 \leq i \leq \ell}$. When not important, the arcs will be omitted. The \defi{length} of an oriented walk is its number of arcs. 
We  also say that $W$ is a $v_0v_\ell$-walk. 

For $v \in V(D)$, its \defi{out-neighbourhood} $N^+(v)$ is the set of vertices~${u \in V(D)}$
such that~${(v,u) \in A(D)}$, and its  \defi{outdegree}  is $\deg^{+}(v):=|N^+(v)|$. The \defi{in-neighbourhood} $N^{-}(v)$ and \defi{indegree} $\deg^{-}(v)$ are defined analogously. A vertex with indegree zero is called a \defi{source}, whereas a vertex with outdegree zero is called a \defi{sink}. The \defi{semidegree} of $v$ is the number $\deg^{0}(v) := \min\{\deg^{+}(v),\deg^{-}(v)\}$.
We denote by $\semidegree(D)$ the \defi{minimum semidegree} of all the vertices
of~$D$. The \defi{maximum outdegree} (resp. in-degree) of all the vertices of $D$ is denoted by~$\Delta^+(D)$ (resp. $\Delta^-(D)$). We define $\Delta(D):=\min\{\Delta^+(D), \Delta^-(D)\}$.

\subsection{Trees}
\label{ssec:trees}

An oriented forest is \defi{antidirected} if each of its vertices is a sink or a source. An antidirected tree is \defi{balanced} if it has as many sinks as sources.
An \defi{oriented rooted tree} is a tree~$T$ with a special vertex $r(T)\in V(T)$, called its \defi{root}. 
If the is an arc between $u$ and $v$ and $u$ lies on the unique path from $v$ to $r(T)$ in $T$, we say that~$u$ is
the \defi{parent} of~$v$. 
The \defi{$i$th level} of $T$ is the set of all vertices at distance $i$ from $r(T)$. 

As mentioned above, our embedding strategy relies on finding a vertex $u\in V(D)$ with many neighbours in distinct parts of $D$, in which we embed a vertex of $T$ that cuts $T$ into adequately sized parts. Using the  folklore Lemma~\ref{lemma:cutting_2}  we show, in Lemma~\ref{cutTintoJs}, that every tree can decomposed into forests of bounded size by cutting it at a special vertex $z$.
Note that in the following three lemmas, $T$ is an undirected tree.
\begin{lemma}[see Observation 2.3 in~\cite{havet}]
  \label{lemma:cutting_2}
  Let~$h\in [0,1]$. Every tree $T$ on $k$ edges has a vertex~${z \in V(T)}$ such
  that every component of~${T - z}$ has at most $hk$ vertices,
  except for at most one, which has at most~$(1-h)k$ vertices.
\end{lemma}

\begin{lemma}
\label{cutTintoJs}
For $h\in [0,1]$, every tree $T$ on $k$ edges has a vertex $z$ such that there is a partition of the components of $T-z$ into a finite family of sets $\{\mathcal{J}_i\}_{i\in I}$ so that~$|\bigcup \mathcal{J}_1| \leq (1-h)k$ and~$|\bigcup \mathcal{J}_i| \leq hk$ for $i\neq 1$.
\end{lemma}

\begin{proof}
  By Lemma~\ref{lemma:cutting_2}, the tree $T$ has a vertex $z \in V(T)$ such that each component of $T - z$ has size at most $hk$, except for at most one, which has size at most~$(1-h)k$. We take two disjoint  subsets of components $\mathcal{J}_1,\mathcal{J}_2$ of total orders at most $(1-h)k$ and $hk$, respectively, with $\mathcal{J}_1$ containing the largest component of $T-z$, such that $|\mathcal{J}_1\cup\mathcal{J}_2|$ is maximum. Then the set $\mathcal{J}_3$ of all remaining components has total order at most $hk$, as otherwise all $\mathcal{J}_1\cup\mathcal{J}_2$ could be assigned to $\mathcal{J}_1$ freeing up $\mathcal{J}_2$ to receive one of the components of $\mathcal{J}_3$, contradicting the maximality of~${|\mathcal{J}_1\cup\mathcal{J}_2|}$. 
\end{proof}

Once the tree is decomposed into forests, we will embed each tree in each forest one at a time.
To embed each of these trees, we will need to cut them into smaller pieces. The next result shows that any tree can be decomposed into small subtrees with few connecting vertices.

\begin{lemma}[Proposition 4.1 in \cite{BPS1}]
  \label{lemma:cutting_1}
  For every $\beta > 0$, and for all rooted trees $T$ on $k$ edges,
  there is a set $S \subseteq V(T)$ and a family $\mathcal{T}$ of disjoint
  rooted trees such that 
  \begin{enumerate}[(a)]
    \item $r(T) \in S$;
    \item $\mathcal{T}$ consists of the connected components of $T - S$, and every $T'\in \mathcal{T}$
          is rooted at the vertex closest to $r(T)$;
    \item $|T'| \leq \beta k$, for every $T' \in \mathcal{T}$; and
    \item $|S| \leq \frac{1}{\beta} + 2$.
  \end{enumerate}
\end{lemma}

After the tree is decomposed, we will have to decide where to embed each of the small trees
in $\mathcal{T}$. For this, we use the following result.

\begin{lemma}[Lemma 3.5 in~\cite{SZ24}]
  \label{lemma:camila}
  Let $s,t \in \NN, \mu > 0$ and $(p,q)_{i \in I}\subseteq \NN^2$ be
  a family such that:
  \begin{enumerate}[(a)]
    \item\label{camila1} $(1-\mu)\sum_{i \in I}p \leq \sum_{i \in I}q \leq (1+\mu)\sum_{i \in I}p$,
    \item\label{camila2} $p + q \leq \mu s$ for every $i \in I$ and
    \item\label{camila3} $\max\left\{\sum_{i \in I}p, \sum_{i \in I}q\right\} < (1-10\mu)st$.
  \end{enumerate}
  Then there is a partition $\mathcal{J}$ of $I$ of size $t$ such that, for every $J \in \mathcal{J}$,
  \[
    \sum_{j \in J}p_j \leq (1-7\mu)m\;\;\text{ and }\;\;\sum_{j \in J}q_j \leq (1-7\mu)m.
  \]
\end{lemma}

\subsection{The diregularity lemma}
\label{ssec:diregularity}

Let~$D$ be a digraph, and let $X,Y \subseteq V(D)$ be nonempty and disjoint. We set
\[
  \text{$d^{+}(X,Y) := \frac{|(X\times Y)\cap A(D)|}{|X||Y|}$ \ and \ $d^{-}(X,Y):=d^+(Y,X)$.}
\]
 For~$\eps > 0$ we say~${X' \subseteq X}$
is~\defi{$\eps$\nobreakdash-significant} if~${|X'| \geq \eps|X|}$.
For~${\diamond \in \{+,-\}}$, the pair~$(X,Y)$
is~\defi{$(\eps,\diamond)$-regular}
if~${|d^{\diamond}(X,Y) - d^{\diamond}(X',Y')| \leq \eps}$ for
all~{$\eps$\nobreakdash-significant} subsets~${X' \subseteq X}$
and~${Y' \subseteq Y}$. If furthermore ${d^{\diamond}(X,Y) \geq d}$ for
some~${d \geq 0}$, we call~$(X,Y)$ \defi{$(\eps,\diamond,d)$-regular}. Given an $(\eps,\diamond,d)$-regular pair $(X, Y)$ and an $\eps$-significant $Y'\subseteq Y$, a vertex $v\in X$ is called $\diamond$-\defi{typical} to $Y$ if $d^\diamond(\{v\},Y)\geq d-\eps$. We  simply call such vertices \emph{typical} when the prefix is implicit. The next well-known fact states that in a regular pair almost all vertices  are typical to any significant subset of the other side.

\begin{fact} [see \cite{komlos2000regularity}] \label{typical}
    Let $(X, Y)$ be an $(\eps,\diamond)$-regular pair and $Y'\subseteq Y$ an $\eps$-significant subset. Then all but at most $\eps|X|$ vertices of $X$ are $\diamond$-typical to $Y'$.
\end{fact}

Szemerédi’s regularity lemma~\cite{szemeredi} states that every large graph can be partitioned into a bounded
number of vertex sets, most of which are pairwise $\eps$-regular. We will use the following analogue of Szemerédi’s result to digraphs, obtained by Alon and Shapira~\cite{alon2003testing}.

\begin{lemma}[Degree form of the diregularity lemma \cite{alon2003testing}]
  \label{lemma:diregularity}
  For every $0<\eps < 1$ and $m_0 \in \NN$, there are integers $M_0$ and $n_0$
  such that the following holds for all $n \geq n_0$ and $d \leq 1$.
  If $D$ is a digraph on $n$ vertices, then there is a partition of $V(D)$
  into $V_0,V_1,\dots,V_r$ and a spanning subgraph~$D'$ of~$D$ such that:
  \begin{enumerate}[(a)]
    \item $m_0 \leq r \leq M_0$,
    \item $|V_0| \leq \eps n$ and $|V_1|=\dots=|V_r|$,
    \item $\deg_{D'}^{\diamond}(v) > \deg_{D}^{\diamond}(v) - (d+\eps)n$,
      for all $v \in V(D)$ and $\diamond \in \{+,-\}$,
    \item $V_i$ is an independent set in $D'$, for all $i \in [r]$, and
    \item for all $1 \leq i,j \leq r$, the ordered pair $(V_i,V_j)$ is
      $(\eps,+)$-regular in $D'$ with density either $0$ or at least $d$.
  \end{enumerate}
\end{lemma}

Let $D$ be a digraph on $n$ vertices, let~${\eps, d < 1}$, and let $V_0,\dots,V_r$ be a partition
of $V(D)$ as in Lemma~\ref{lemma:diregularity}. An \defi{$(\eps,d)$-reduced digraph} $R$
of $D$ with respect to $V_0,\dots,V_r$ is the digraph with vertex set~${V(R) := \{V_1,\dots,V_r\}}$, and arc set $A(R):=\{(V_i, V_j)|\text{ the pair~$(V_i,V_j)$ is~$(\eps,+,d)$-regular}\}$. We also refer to the vertices of $R$ as \defi{clusters}. In the embedding procedure we will mainly work with the reduced graph, since it inherits important properties from $D$.

\begin{fact} [see \cite{kelly2010cycles}]
  \label{fact:semidegree_reduced} 
  Let $0 < 3\eps \leq \eta \leq \alpha/2$. If~$D$ is a digraph on~$n$ vertices
  with~${\semidegree(D) \geq \alpha n}$ and~$R$ is
  an~$(\eps, \eta)$-reduced digraph of~$D$,
  then $\semidegree(R) \geq (\alpha - 2\eta)|R|$. Moreover, if $D$ is an oriented graph, then $R$ can be assumed to also be oriented with $\delta^0(R)\geq (\alpha -4\eta)|R|$.
\end{fact}

\section{Anticonnected Components}
\label{sec:anticomponents}

Let  $W=(v_1,v_2,\dots, v_\ell)$ be a walk on a digraph. 
For $1\le i<\ell$, we say the edge $v_iv_{i+1}$ is directed \defi{forward} if it is directed from $v_i$ to $v_{i+1}$, and \defi{backward} otherwise. 
We call $W$ an \defi{antiwalk} if it alternates between forward and backward edges.
A vertex $v$ is called an \defi{out-vertex} (resp.~\defi{in-vertex}) of~$W$ if it has positive out-degree (resp.~in-degree)  in the graph induced by the arcs of~$W$. Note that if~$W$ is not a path, then $v$ may be both an out-vertex and  an in-vertex of~$W$. 
If $v_1$ is an out-vertex and $v_\ell$ is an in-vertex, then $W$ is an \defi{out-in antiwalk}. We define \defi{out-out, in-out}, and \defi{in-in} antiwalks analogously. The trivial antiwalk on one vertex is always considered to be both out-out and in-in (and not considered to be out-in or in-out). 

A subdigraph $C$ of a digraph $D$ is an \defi{anticonnected component} of $D$ 
if there is an $a\in V(D)$ such that  $C$ consists of all vertices and arcs that lie on
some out-in or out-out antiwalk starting at~$a$.  Note that then  $V(C)$ can be covered by two sets $\In(C)$ and $\Out(C)$ such that for each pair $u\in\Out(C)$, $v\in \In(C)$ there is an out-in antiwalk from $u$ to $v$. Namely, the set $\In(C)$ ($\Out(C)$) is the set of all vertices $v$ such that there is an out-in (out-out) walk from $a$ to $v$. Furthermore, all arcs of $C$ go from $\Out(C)$ to $\In(C)$. It is possible that $\In(C)\cap\Out(C)\neq\emptyset$.  
Also, observe that if $C_1$, $C_2$ are distinct anticonnected components of~$D$ then they do not share arcs and $\Out(C_1)\cap\Out(C_2)=\emptyset=\In(C_1)\cap\In(C_2)$.
Finally, call a digraph~$D$ \defi{anticonnected} if it is an anticonnected component of itself.  

Different from walks in connected graphs, the number of vertices in a shortest antiwalk between a pair of vertices in an anticonnected digraph $D$ can be greater than $|V(D)|$. Indeed, consider for example this anticonnected digraph $D$:
\raisebox{-0.3em}{\tikz[node distance=9pt]{
    \node[draw,shape=circle,fill=black,inner sep=0pt,minimum size=5pt] (u) at (0,0) {};
    \node[draw,shape=circle,fill=black,inner sep=0pt,minimum size=5pt] (x) at (0.8,0) {};
    \node[draw,shape=circle,fill=black,inner sep=0pt,minimum size=5pt] (y) at (1.6,0) {};
    \node[draw,shape=circle,fill=black,inner sep=0pt,minimum size=5pt] (v) at (3.2,0) {};
    \draw[thick,->-=.8] (u) -- (x);
    \draw[thick,->-=.8] (x) -- (u);
    \draw[thick,->-=.8] (x) -- (y);
    \draw[thick,->-=.8] (y) -- (x);
    \draw[thick,->-=.9] (y) -- (v);
    \draw[thick,->-=.9] (v) -- (y);
    \draw[thick,->-=.5] (u) to [bend left=23] (v);
    \draw[thick,->-=.5] (u) to [bend left=23] (v);
    \node[fill=white, inner sep=1pt] at (2.4,0) {$\dots$};
    \node[inner sep=0pt, above left of=u] {$v_1$};
    \node[above right of=v] {$v_{2k+1}$};}},
    where $k$ is a positive integer.
Notice that the number of vertices in a shortest in-out $v_1v_{2k+1}$-antiwalk in this graph has $2|V(D)|$ vertices. However, as it is shown in the next proposition, the number of vertices in a shortest antiwalk cannot be arbitrarily large.

\begin{prop}
  \label{prop:adiam}
  Let~$D$ be an anticonnected digraph. Let $\diamond\in\{\text{out, in}\}$, and let~${u,v\in V(D)}$ be distinct, such that there is an $\diamond$-out $uv$-antiwalk in~$D$. Then there exists an $\diamond$-out $uv$-antiwalk of length at most~$2|V(D)|$ in $D$.
\end{prop}

\begin{proof}
  Assume $\diamond =\text{out}$ (the other case is analogous). Let $u,v \in V(D)$ with $u\neq v$ and let~$W$ be a shortest out-out $uv$-antiwalk
  in $D$. Suppose that~${|W| \geq 2|V(D)|+1}$. By the pigeonhole principle, there is a vertex $x$ that appears at least three times in $W$. Without loss of generality, assume that~$x$ appears at least twice as an out-vertex in $W$. Then $W = W_1 x W_2 x W_3$, where $xW_2x$ is a non-trivial out-out-walk. Since the out-out $uv$-antiwalk $W' = W_1 x W_3$ is shorter than~$W$, we reach a contradiction. 
\end{proof}

The next lemma relates the order of $|\In(C)|$, $|\Out(C)|$ and the outdegree of $\In(C)\cap\Out(C)$ with the minimum semidegree of~$D$, where $C$ is an anticonnected component  in a digraph~$D$. 

\begin{lemma} \label{lem:sizes}
  Let  $C$ be an anticonnected component of a digraph $D$ with $\delta^0(D)>0$, and let $B$ be the subdigraph of $D$ induced by $\In(C)\cap \Out(C)$.
  Then
  \begin{enumerate}[label=(\roman*)]
      \item $\min\big\{|\Out(C)|, |\In(C)|\big\}\ge\delta^0(D)$ and 
      \item 
      $\delta^+(B)
      \ge 2\delta^0(D)-|V(C)|$.    
  \end{enumerate}
\end{lemma}
 
\begin{proof}
Since $\delta^0(D)>0$, we know that $C$ has an arc, say  $uv$. Consider the outdegree of $u$ and the in-degree of $v$ to see that $(i)$ holds.

Consider the total out-degree of any vertex $v$ having out-degree $\delta^+(B)$ into~$B$ (this may be~$0$) to see that $|\In(C)|-|V(B)|\ge\delta^0(D)-\delta^+(B)$.
Thus, by $(i)$, and since $|V(C)| = |\In(C)|+|\Out(C)|-|V(B)|$, it follows that $\delta^+(B)\ge\delta^0(D)-\In(C)|+|V(B)|\ge 2\delta^0(D)-|V(C)|$, as desired for $(ii)$.
\end{proof}
 
An \defi{anticonnected matching}, for brevity \defi{antimatching}, of \defi{size}  $t$ in an
anticonnected digraph $D$ is a collection of disjoint
arcs~${M=\{(a_1,b_1),\dots,(a_t,b_t)\}}$ of $D$.
We set $\In(M):=\{a_i:i \in [t]\}$, 
$\Out(M):=\{b_i:i \in [t]\}$ and $V(M):=\In(M)\cup\Out(M)$. We finish this section by relating the size of an antimatching in a large anticonnected component with the degree of its vertices.

\begin{lemma}
  \label{lemma:matchings}
  Let $t \in \NN^{+}$, and let $D$ be an anticonnected digraph on
  at least~$2t$ vertices. If~${\deg^{+}(v) \geq t}$
  for all $v \in \Out(D)$ and~${\deg^{-}(v) \geq t}$ for all $v \in \In(D)$
  then $D$ has an antimatching of size  $t$.
\end{lemma}

\begin{proof}
Consider the underlying graph $G$ of $D$. As $\delta(G)\ge t$,  there is a
matching  $M$ of size~$t$ in $G$ (this can be seen by taking any matching of maximum size and considering the neighbourhoods of two uncovered vertices). Viewing $M$ as a matching in $D$, we found the desired antimatching. 
\end{proof}

\section{Embedding Trees in a Large Anticonnected Component}
\label{sec:embedding}  

We start with a lemma from~\cite{SZ24}.
Call an antidirected tree $T$ \defi{consistent} with an antiwalk $W$  if the following holds: 
$r(T)$ is a source if and only if $W$ starts with an out-vertex. 

\begin{lemma}\cite[Lemma 3.9]{SZ24}
  \label{prop:connect}
  Let $\eps \in (0,\frac 14)$, $d \geq 2\sqrt{\eps}$\footnote{This result is stated in~\cite{SZ24} only for $d=2\sqrt{\eps}$, but one can easily verify that it holds for every $d \geq 2\sqrt{\eps}$.}, and $s, \ell \in \NN$. Let $W = V_1,\dots,V_{\ell}$ 
  be an antiwalk in an~${(\eps, d)}$-reduced 
  directed graph $R$ of a digraph $D$, where clusters have
  size $s$. For~${i \in [\ell]}$, let $Z_i \subseteq V_i$, $X_{\ell-1}\subseteq V_{\ell-1}$ and $X_\ell\subseteq V_\ell$ be sets of size at least
  $3\sqrt{\eps}s$. Let $T$ be an antidirected rooted tree on at most $\frac{\eps}{10}s$ vertices that is consistent with $W$.
  
  Then there is an embedding of $T$ into $W$ so that:
  
  \begin{itemize}
      \item for $i\leq \ell-2$, the vertices at the $i^{th}$ level of $T$ are embedded in  vertices of $Z_i$ that are typical with respect to $Z_{i+1}$;
      \item for  $i\ge\ell-1$ and $i-\ell$ odd, the $i$th level of $T$ is embedded in  vertices of $X_{\ell-1}$ that are typical with respect to both $X_\ell$ and $Z_\ell$;  
       \item for  $i\ge\ell$ and $i-\ell$ even, the $i$th level of $T$  is embedded in  vertices of $X_{\ell}$  that are typical with respect to both $X_{\ell-1}$ and $Z_{\ell-1}$.  
  \end{itemize}

\end{lemma}

We are now ready for our main embedding lemma.

\begin{lemma}
  \label{lemma:emb_large}
  For every $\gamma, c>0$ there are $\varepsilon>0$ and $n_0\in\mathbb N$ such that 
  the following holds for every~${k,n}$ with~${n\ge n_0}$, and~${k\ge\gamma n}$, every $k$-arc oriented tree~$T$ with~${\Delta(T)\leq (\log k)^c}$ and every~$n$-vertex digraph~$D$. Let~$R$ be an~$(\varepsilon, 100\sqrt\varepsilon)$-reduced digraph of $D$ with clusters of size $s$, and let $k_R=\frac{k}{n}\cdot|R|$.  Let $C$ be an anticonnected component of $R$ and let $B$ be the subdigraph of~$R$ induced by $\In(C)\cap \Out(C)$. 
  \begin{enumerate}[(a)]
    \item If $\semidegree(R)\geq (\frac{1}{2}+\gamma)k_R$ and  $|C|\ge (1+\gamma)k_R$, then~$T$ embeds in $D$.
    \item  Let $U\subseteq  V(D)\setminus \bigcup V(B)$ such that $|W\setminus U|\ge 9\sqrt{\eps}s$, for every $W\in V(C)$. Let $N\subseteq \bigcup \In(C)\setminus U$ with~${|N|\ge \frac\gamma{100}|\bigcup \In(C)|}$. Let $J\subseteq T$ 
be an antidirected forest whose roots are sinks such that
$$|J|\le \max\Big\{\delta^0(R)s-|U|, \delta^+(B)s\Big\} - \frac{\gamma}{100}k .$$ 
Then, 
$J$ embeds in $\bigcup V(C)-U$, with its roots embedded in $N$, and such that at least~$9\sqrt{\eps}s$ vertices of every $W\in V(C)$ are neither used nor in $U$. Further,  the embedding uses at least $\min\{|J|, \delta^+(B)s - \frac{\gamma}{100}k
\}$ vertices of  $\bigcup V(B)$. 
  \end{enumerate}
\end{lemma}

\begin{proof}
Let $\eps\ll\gamma$, let $\beta := \frac{\eps(1-\eps)(\sqrt{\eps}-8\eps)}{M_0}$ and let $M_0,n_0'$ be given by Lemma~\ref{lemma:diregularity} for input $\eps$ and $m=\frac 1\eps$. Choose $n_0 \in \NN$ such that $n_0\ge n'_0$ and 
  \begin{equation}\label{choicen_0}
         \frac{2}{\beta}\left(\log n\right)^{2cM_0+1} 
         < \frac{\eps s}2.
  \end{equation}
  for all $n \geq n_0$. 

 Let us first show $(a)$.
 We embed~$T$ into~$C$ in four steps: cutting~$T$ into pieces, slicing the clusters, finding an antimatching and assigning the pieces, and finally embedding~$T$.

  \textit{Step 1. (Preparing $T$.)} By Lemma~\ref{lemma:cutting_1} (applied to the underlying undirected tree), $T$ can be decomposed into a set $S\subseteq V(T)$ of size at most $\frac{2}{\beta}$ and a family $\mathcal{T}$ of connected components of $T-S$, each of order at most $\beta k$. We  further decompose each $X\in\mathcal{T}$ into a tree $L_X$, consisting of the first $2M_0$ levels of $X$, and the set $P_X$ of connected components of $X-L_X$.  
  Let $\mathcal{L}:=\{L_{X}: X\in\mathcal{T}\}$, and let $\mathcal{P}:=
  \bigcup_{X\in\mathcal T}P_X$.
  We say that the elements of $S$, $\mathcal{T}$, $\mathcal{L}$ and $\mathcal{P}$ are the \defi{seeds}, \defi{parts}, \defi{links} and \defi{pieces} of $T$, respectively.
  
  Given that~$|\mathcal{L}|\leq\frac{2}{\beta}\Delta(T)$, we can use~\eqref{choicen_0} to see that
  \begin{equation}
    \label{eq:link}
    \sum_{L \in \mathcal{L}}|L| < |\mathcal{L}| \Delta(T)^{2M_0+1} 
    <\frac{\eps s}2.
  \end{equation}

  \textit{Step 2. (Slicing the clusters.)} In each cluster~$V_i\in V(C)$ we reserve a part to accommodate the links. For this, we arbitrarily
  partition each $V_i$ into two sets $V_i^S$  and $V_i^{P}$ with $|V_i^{S}|  = 10\sqrt{\eps}s$. We  call these sets the $S$-slice and $P$-slice of the corresponding cluster.
  Our plan is to embed   $S\cup \bigcup\mathcal{L}$ into the $S$-slices, and to embed $\bigcup\mathcal{P}$ into the $P$-slices of the clusters of $C$.

   \textit{Step 3. (Finding an antimatching and assigning the pieces.)} 
  Since $N^+(V_i)\subseteq V(C)$ for each $V_i\in \Out(C)$, we know that ${\deg_C^{+}(V_i) \geq (\frac{1}{2} + \gamma) k_R}$ for all $V_i \in \Out(C)$ and similarly,
  $\deg_C^{-}(V_i) \geq (\frac{1}{2} + \gamma) k_R$ for all $V_i \in \In(C)$.
  Also, $|V(C)| \ge  2(\frac{1}{2} + \gamma) k_R$.
  By Lemma~\ref{lemma:matchings} with $t := (\frac{1}{2}+\gamma)k_R$,
  there is an antimatching~${M:=\{(A_1,B_1),\dots,(A_t,B_t)\}}$ in~$C$
  of size~$t$.
  
  We now show that there exists an assignment of the pieces, that is, there is a partition $(\mathcal T_e)_{e\in M}$ of~$\mathcal T$ so that for each $e=(A,B)\in M$, the pieces in $\mathcal{T}_e$ contain in total at most $(1-7\sqrt{\eps})|A^P|$ sources and at
  most $(1-7\sqrt{\eps})|B^P|$ sinks. Note that although we do not plan to embed the links into the set $A^P\cup B^P$, we partition $\mathcal T$ and not $\mathcal P$ because this will have the effect that all pieces with the same link lie in the same partition class $\mathcal T_e$.
  To show that $(\mathcal T_e)_{e\in M}$ exists, 
  we apply Lemma~\ref{lemma:camila}
  with $(p_X,q_X)_{X \in \mathcal{T}}$, $\mu = \sqrt{\eps}$, $s$ and~$t$, where for $X\in\mathcal{T}$, we let $p_X$ and $q_X$ denote its number of sources and sinks, respectively.
  It is easy to verify that the three conditions of Lemma~\ref{lemma:camila}
  hold. 

  \textit{Step 4. (Embedding $T$.)}
  For elements $X, Y\in S\cup\mathcal{T}$, we say that $X$ is the \defi{parent} of $Y$ if $X$ is or contains the parent of $Y$ or $r(Y)$. 
  We start the embedding with the element of $S \cup \mathcal{T}$ that contains $r(T)$ and then at each step, embed an element  of $S \cup \mathcal{T}$ whose parent is already embedded, until we have embedded all of $T$. Throughout the embedding, we will ensure that if $v$ is the image of a source (sink) of $T$ and $v$ belongs to cluster $W$, then
  \begin{equation}\label{typi}
    \text{$v$ is $+$-typical ($-$-typical)  to $U^S$ for some $U\in \In(C)\cap N^+(W)$ ($\Out(C))\cap N^-(W)$).}
  \end{equation}

Assume we are about to embed a seed $s$.
If $s\neq r(T)$ then by~\eqref{typi},  the parent of $s$ is already embedded in some vertex $w$ of $W$ that  is typical to the $S$-slice $U^S$ of some  $U\in \In(C)\cap N^+(W)$ (resp.~$\Out(C)\cap N^-(W)$), and we embed $s$ in any neighbour of $w$ fulfilling~\eqref{typi}. If $s=r(T)$ is a source (sink) of $T$ we choose any $U\in\In(C)$ ($U\in\Out(C)$) and embed $s$ in any vertex of $U^S$ fulfilling~\eqref{typi}. Such an embedding is possible by Fact \ref{typical}, and since $|S\cup\bigcup\mathcal L|<\eps s$ by~\eqref{eq:link}.

  Now say we are embedding a part $X\in\mathcal{T}$. First assume $r(X)$ is a sink and is different from $r(T)$. By~\eqref{typi}, the parent of $r(X)$ is already embedded in some $v$ that is $+$-typical to $U^S$ for  some~$U\in \In(C)\cap N^+(W)$ where $W$ is the cluster that hosts $v$. 
  Let $e=(A,B)$ be such that $X\in \mathcal{T}_e$.  By Proposition \ref{prop:adiam}, there is an in-out antiwalk~$W'$   
  of length at most $2M_0$
  from $U$ to $A$.  In fact $W'$ has length at most $2M_0-1$ since it is odd. We obtain a $U$--$B$ walk $W$ of length exactly $2M_0$ from $W'$ by adding the edge $AB$ at the end and traversing it back and forth the appropriate number of times, ending in $B$.
  We apply Lemma~\ref{prop:connect} to embed $\mathcal L_X$ into the $S$-slices of the clusters on  $W$ and $\bigcup \mathcal P_x$ into the $P$-slices of $A$ and $B$. 
  
  If $r(X)=r(T)$ is a sink, we proceed as before, but now embed $r(X)$ into the $S$-slice of any cluster $U$. If $r(X)$ is a source,  we proceed as in the case when $r(X)$ is a sink, the only difference is that $W'$ will have even length, and thus can be made to have length exactly $2M_0$ while ending in $A$ (if necessary we traverse $AB$ back and forth several times). This finishes the proof of part $(a)$ of the lemma.

 We now prove part $(b)$ of the lemma. 
 We proceed similarly as in the proof of part $(a)$, and will discuss the differences following the step-by-step procedure from the previous proof. 
 We execute step 1 with $J$ instead of $T$, which means that we  apply Lemma~\ref{lemma:cutting_1} to each of the components of $J$, calling the union of all seeds $S$, and letting $\mathcal T$ denote the set of components of~$T-S$. 
 
 In step 2, we note that by the assumptions of the lemma, there is a cluster $V_N$ in $\In(C)$ that contains at least $2\sqrt\varepsilon s$ vertices of $N$. We take care to slice $V_N$ so that its $S$-slice contains a set $N^*$ of $2\sqrt\varepsilon s$ vertices of $N$. Also, we choose the $S$-slice of each cluster minimising intersection with $U$. 
 We omit step 3.

 For step 4, the actual embedding, we proceed as in part $(a)$, following an embedding order that respects parentage. It does not matter that $J$ is a forest, as this implies fewer restrictions. It is no problem that we need to embed the roots of $J$ into $N$, as we can use $N^*$ for their images. 

  We will embed the seeds and the links as before into the $S$-slices, and the pieces into the $P$-slices. When we embed into an $S$-slice, we need to avoid $N^*$, which is possible as $N^*$ is small. 

   We embed any seed as in~$(a)$.
   Now suppose that we are about to embed a part $X$ from $\mathcal T$. First assume we have used less than $b:=\delta^+(B)s - \frac{\gamma}{100}k$ vertices of  $\bigcup V(B)$ so far. We will embed the pieces of $X$ in some arc $W_1W_2$ in $B$ with a sufficient number of free vertices.
   Since $|B|\ge \delta^+(B)$ and we have  embedded less than
   $b\le(1-\frac{\gamma}{100})|B|s$
    vertices in~$\bigcup V(B)$,   there is a cluster  $W_1\in B$ with at least $\frac{\gamma}{100}s$ free vertices in its $P$-slice. Similarly,   there is an out-neighbour $W_2$ of $W_1$ with at least $\frac{\gamma}{100}s$ free vertices in its $P$-slice.

   Now assume we have already used at least $b$ vertices of  in $\bigcup V(B)$. We will embed the pieces of $X$ in an arc $W_1W_2$ such that $\min\{|W_1^P\setminus U^*|, |W_2^P\setminus U^*|\} \geq  \frac{\gamma}{100}s$, where $U^*$ is the union of $U$ and the vertices used for $J$ so far.   If there is no  cluster $W_1$ with at least $\frac{\gamma}{100}s$ free vertices in its $P$-slice,
    then by  Lemma~\ref{lem:sizes}~$(i)$, and by our assumptions in case $(b)$,
    \[
      |U^*| \ge \Big(1-\frac{\gamma}{100}\Big) |\Out(C)|s 
      > \delta^0(R)-\frac{\gamma}{100} k_R \ge |{J}|+|U|
    \]
    vertices, which is impossible. A similar calculation shows that $W_1$ has a neighbour $W_2\in \In(C)$ with $|W_2^P\setminus U^*|\geq \frac{\gamma}{100}s$. This finishes the proof of part $(b)$.
  \end{proof}

Note that part $(a)$ of the previous lemma  already gives the result from \cite{SZ24}: 

\begin{cor}\label{corollo}
  For every $\gamma>0$ and $c \in \NN$, there exists $n_0\in\NN$ such that for all
  $n \geq n_0$ and $k\ge \gamma n$, every oriented graph $D$ on
  $n$ vertices with $\semidegree(D) \geq (\frac{1}{2} + \gamma)k$ contains
  every balanced antidirected tree~$T$ with~$k$ arcs and with
  $\Delta(T) \leq (\log k)^{c}$.
\end{cor}

\begin{proof}
Applying Lemma~\ref{lemma:emb_large} with  $\frac{\gamma}{100}$ and $c$ gives $\eps$ and $n'_0$,
and applying Lemma~\ref{lemma:diregularity} with $\eps$ and $m_0=\frac 1\eps$ gives $n''_0$ and~$M_0$. Set $n_0:=\max\{n'_0, n''_0\}$. 
By Lemma~\ref{lemma:emb_large}~$(a)$, we only need to find an anticonnected component $C$ of the reduced oriented graph $R$ of $D$ such that $|C|\ge (1+\gamma)k_R$ (where $k_R$ is defined as in Lemma~\ref{lemma:emb_large}). Consider any anticonnected component $C$.
By Fact \ref{fact:semidegree_reduced} and by Lemma~\ref{lem:sizes}~$(a)$, $C$ has the desired size unless there is a cluster $U\in\In(C)\cap \Out(C)$. So assume this is the case.  Note that $N^+(U)\subseteq\In(C)$ and $N^-(U)\subseteq\Out(C)$, and since $D$ is an oriented graph, $N^+(U)\cap N^-(U)=\emptyset$. So $C$ has the desired size. 
\end{proof}

\section{Proof of Theorem \ref{tree}}\label{sec:three}

This section is devoted to the proof of Theorem \ref{tree}.
Given $\gamma$, $c$ and $\ell$,  we apply Lemma~\ref{lemma:emb_large} with input $\frac{\gamma}{100 \ell},c$ to obtain $\eps$, and $n'_0$.
We then apply Lemma~\ref{lemma:diregularity} with $\eps$ and $m_0=\frac 1\eps$ to obtain $M_0$ and $n''_0$. We set $n_0:=\max\{n'_0, n''_0\}$.  
  
Given $n$, $k$, a $D$ and $T$, we apply Lemma \ref{cutTintoJs} to the underlying undirected tree of $T$. This yields  a vertex $z\in V(T)$ and disjoint forests $J_1, J_2, \ldots, J_\ell$ partitioning $T-z$ such that $$j_1:=|J_1|\le \Big(\frac{\ell}{2\ell -1}\Big)k\hspace{2mm} \text{and}\hspace{2mm} j_i:=|J_i|\le \Big(\frac{\ell-1}{2\ell -1}\Big)k\hspace{2mm} \text{for}\hspace{2mm} 2\le i\le\ell.$$ 
  We may assume that $j_i\geq j_{i+1}$ for all $1\leq i\leq \ell-1$. We also assume that $z$ is an out-vertex, for if this is not the case then we swap all orientations in $T$ and in $D$ and note that any embedding in this new setting corresponds to an embedding in the original setting. Observe that vertex $z$ being an out-vertex implies that the roots of the forests~$J_i$ are all sinks. 
  
Choose any $u \in V(D)$ with $\deg^+(u) \geq \Delta(D)$.
Let $R$ be an $(\eps, 100\sqrt{\eps})$-reduced directed graph  of $D-u$, with clusters of size $s$. Set $k_R:=\frac{k}{n}\cdot |R|$. By Fact~\ref{fact:semidegree_reduced} with $a:=(\frac{\ell-1}{2\ell-1}+\gamma)\frac{k}{n}$, we have
 \begin{equation}\label{deltaR2}
     \delta^0(R)\ \ge \ \left(1-\frac{\gamma}{100\ell}\right)\delta^0(D)\frac{k_R}k \ \ge \ \left(\frac{\ell}{2\ell-1}+\frac\gamma2\right)k_R. 
 \end{equation}
In particular, $\delta^0(R) \ge (\frac 12 +\frac{\gamma}{100\ell}) k_R$, and therefore,
   if $R$ has an anticonnected component of
  size~at least $(1+\frac{\gamma}{100\ell})k_R$, then we are done
  by Lemma~\ref{lemma:emb_large} $(a)$.
  So we can assume that for every anticonnected component $C$ of $R$ we have
  \begin{equation}\label{compsaresmall2}
      |C| < \left(1+\frac{\gamma}{100\ell}\right)k_R.
  \end{equation}

Let $\mathcal C$ be the set of all  anticonnected components $C$ of $R$ such that $z$ sends at least $\frac{\gamma}{100}|C|s$ arcs to $\In(C)$. We will embed each forest $J_i$ into a different component of $\mathcal{C}$. To begin, we must show that there are enough components in $\mathcal{C}$ to accommodate $T-z$ in this manner. Indeed, since  distinct anticonnected components~$C$ of~$R$ have pairwise disjoint sets $\In(C)$, vertex $u$ sends fewer than $\frac{\gamma}{100}n$ arcs to clusters in $R\setminus\bigcup\mathcal C$. Also, $u$ sends at most $|\bigcup \mathcal C|s$ edges to $\bigcup\mathcal  C$. Since $d(u)\ge (1+\gamma)(\ell-1) k$, and because of~\eqref{compsaresmall2}, we deduce that
\begin{equation}\label{goodCs2}
      \text{
  $|\mathcal C|\ge \ell$.}
\end{equation}
  
Furthermore, for $C\in\mathcal C$, let $B_C$ be the subdigraph of $C$ induced by $\In(C)\cap \Out(C)$.  
We note that 
\begin{equation}\label{Bgood2}
    V(B_C)\cap V(C')=\emptyset
\end{equation}
if $C, C'\in\mathcal C$ are distinct. Finally, from Lemma~\ref{lem:sizes}  we get that $\min\{|\Out(C)|, |\In(C)|\}\geq \delta^0(R)$, while $\delta^+(B_C)\ge 2\delta^0(R)-|C|$ for every $C\in\mathcal C$.
Observe that 
\[
    b:=\delta^+(B)s-\frac\gamma{100}k\ge (2\delta^0(R)-|C_1|)s-\frac\gamma{100}k,
\]
where we employed~\eqref{deltaR2} for the inequality.
    
We  proceed to embed $T$. We start by embedding $z$ in $u$. Since $j_1\le\delta^0(R)s-\frac{\gamma}2k$, we can apply Lemma~\ref{lemma:emb_large} $(b)$ to embed $J_1$ into an arbitrarily chosen component $C_1\in\mathcal C$, with $U=\emptyset$ and $N$ being the set of out-neighbours of $u$ in~$C_1$. The lemma guarantees that  that $N$ receives the roots of $J_1$, and if any vertices outside $B_{C_1}$ are used for the embedding, then also at least $b$ vertices of~$B_{C_1}$ are used. 
By~\eqref{Bgood2}, these vertices do not belong to any of the other antidirected components in $\mathcal C$. 

Let  $U_1$ be the set of all images of vertices of $J_1$ in $C_1-B_{C_1}$. Note that $|U_1|\le j_1-b$. Of course, $U_1$ may intersect some components of $\mathcal{C}$ other than $C_1$. However, as  $|\mathcal C|\ge \ell$ by~\eqref{goodCs2}, we can choose $C_2\in\mathcal C\setminus\{C_1\}$ so that its clusters contain at most 
 \begin{equation}\label{U1}
 \frac{|U_1|}{\ell-1}\le\frac{j_1-b}{\ell-1}\le \Big(\frac{|C_1|-\delta^0(R)}{\ell -1}\Big)s-\frac\gamma{10\ell}\le \left(\frac{2\ell-1-\ell}{(2\ell-1)(\ell -1)}\right)k-\frac\gamma{10\ell}\le \Big(\frac{1}{2\ell-1}\Big)k-\frac\gamma{10\ell}
 \end{equation}
 vertices of $U_1$.  
 We use Lemma~\ref{lemma:emb_large} $(b)$ to embed $J_2$ into   $\bigcup V(C_2)\setminus U_1$, where $N$, similarly as before, consists of the out-neighbours of $u$ in~$\In(C_2)$. Note that we can apply the lemma to $J_2$, since  by~\eqref{U1}, 
 \[
    |J_2|+\left|U_1\cap \bigcup V(C_2)\right|\le
    \Big(\frac{\ell-1}{2\ell-1}\Big)k+\frac{|U_1|}{\ell-1}
    \le\delta^0(R)s-\frac{\gamma}{100 \ell} k.
  \]
Lemma~\ref{lemma:emb_large} $(b)$  guarantees that  if any vertices outside~$B_{C_2}$ are used in the embedding, then also   at least $b$ vertices of~$B_{C_2}$ were used. 

If $\ell=2$, we are done. So  assume  that $\ell>2$. We will inductively embed each forest $J_i$ into some component $C_i$ of $\mathcal C$, for $3\leq i\leq \ell$. Suppose that we are in step $i$, and let $C_j$, $j<i$, be the components used before step $i$. Let $U_{i-1}$ be the set of all vertices of $D\setminus \bigcup_{1\le j<i} \bigcup V(B_{C_j})$ used before step $i$ for the embedding, and choose $C_i$ to be the component of $\mathcal C\setminus\{C_1, C_2,\ldots, C_{i-1}\}$ whose clusters intersect $U_i$ the least. Then
 \begin{equation}\label{eq10}
    u:=\left|U_{i-1}\cap \bigcup V(C_i)\right|\le \frac{|U_{i-1}|}{\ell-i+1}.
 \end{equation}
 If $j_i\leq b$ then we can embed $J_i$ in $\bigcup B_{C_i}$ by Lemma~\ref{lemma:emb_large}~$(b)$ and we are done. So, $j_i>b$ and thus $j_j>b$ for all $j<i$, by our convention that $j_{i+1}\le j_i $ for all $i\le\ell$. Therefore 
\begin{equation}\label{eq11}
    |U_{i-1}|\le j_1+j_2+\cdots +j_{i-1}-(i-1)b\le k-j_i-(i-1)b.
\end{equation}
In order to use Lemma~\ref{lemma:emb_large}~$(b)$ to embed $J_i$ in $C_i$, we need to show that $u+j_i\le \delta^0(R)s-\frac{\gamma}{100\ell} k$. Indeed, by (\ref{eq10}) and (\ref{eq11}), and since $1\le\ell-i$ and $b
  \ge (\frac 1{2\ell -1}+\frac{\gamma}2)k$, and also $j_i\le \frac ki$, we have
\begin{align*}
    u+j_i& 
\le \frac{k-j_i-(i-1)b}{\ell-i+1} + j_i
\\& \le \Big(\frac{\ell-i}{\ell-i+1}\Big)j_i +\left(\frac{1-\frac{i-1}{2\ell -1}-(i-1)\frac{\gamma}2}{\ell-i+1}\right)k\\ 
&\le \left(\frac{\ell-i}{(\ell-i+1)i}\right)k + \left(\frac{2\ell -i}{(2\ell -1)(\ell-i+1)}\right)k-\Big(\frac{\gamma}{\ell-2}\Big)k
\\ & \le\delta^0(R)s-\Big(\frac{\gamma}{100  \ell}\Big)k,
\end{align*}
where the last inequality follows from the fact that 

$$
\frac{\ell-i}{(\ell-i+1)i} +\frac{2\ell -i}{(2\ell -1)(\ell-i+1)} \le \frac{\ell}{2\ell-1}
$$
which holds for $3\le i\le \ell$ as the reader can check.

We can therefore use Lemma~\ref{lemma:emb_large}~$(b)$ to embed $J_i$ into $C_i$. Doing so for all $i\le\ell$, we finish the embedding and thus the proof.

\bibliographystyle{plain}
\bibliography{antitrees}

\end{document}